\documentclass{amsart}
\usepackage{amsmath,amssymb,amsfonts}
\usepackage{mathrsfs,latexsym,amsthm,enumerate}
\usepackage{tikz}
\usepackage{amscd}
\usepackage[all]{xy}
\usepackage{verbatim}

\newtheorem{lemma}{Lemma}[section]
\newtheorem{proposition}[lemma]{Proposition}
\newtheorem{corollary}[lemma]{Corollary}

\newtheorem{remark}[lemma]{Remark}

\newtheorem{theorem}[lemma]{Theorem}
\newtheorem{example}[lemma]{Example}

\begin{document}

\title[Tarski monoids]{Tarski monoids: \\Matui's spatial realization theorem}

\author{Mark V. Lawson}

\address{Department of Mathematics
and the
Maxwell Institute for Mathematical Sciences,
Heriot-Watt University,
Riccarton,
Edinburgh~EH14~4AS,
United Kingdom}  
\email{m.v.lawson@hw.ac.uk }
\thanks{The author was partially supported by an EPSRC grant  (EP/I033203/1).
He is grateful to Collin Bleak and the other organizers for inviting him to the
{\em Workshop on the extended family of R. Thompson's groups} held at St. Andrews in May 2014
where the preliminary results of this paper were first presented.}

\begin{abstract} We introduce a class of inverse monoids, called Tarski monoids, that can be regarded as non-commutative generalizations of the unique countable, atomless Boolean algebra.
These inverse monoids are related to a class of \'etale topological groupoids
under a non-commutative generalization of classical Stone duality
and, significantly, they arise naturally in the theory of dynamical systems as developed by Matui.
We are thereby able to reinterpret a theorem of Matui (\`a la Rubin) on a class of \'etale groupoids
as an equivalent theorem about a class of Tarski monoids:
two simple Tarski monoids are isomorphic if and only if their groups of units are isomorphic.
The inverse monoids in question may also be viewed as countably infinite generalizations of finite symmetric inverse monoids.
Their groups of units therefore generalize the finite symmetric groups and include amongst their number the Thompson groups $V_{n}$.
 \end{abstract}



\maketitle

\section{Introduction}

It is a theorem of Tarski \cite[Chapter 16]{GH} that any two countably infinite atomless Boolean algebras are isomorphic.
For the purposes of this paper, it is convenient to refer to such a Boolean algebra as the {\em Tarski algebra}.
Under classical Stone duality, the Tarski algebra corresponds to the {\em Cantor space} this time a topological structure with its own uniqueness property.
The Tarski algebra has non-commutative generalizations to what we call {\em Tarski monoids}, a class of inverse monoids,
and under a non-commutative generalization of Stone duality these Tarski monoids are paired with a class of \'etale topological groupoids whose spaces of identities are Cantor spaces;
specifically, the groupoids studied by Matui \cite{Matui13}.
We shall call these {\em Cantor groupoids}.
In his paper, Matui associated a group with each Cantor groupoid called its {\em topological full group}
and proved that when the groupoid is minimal and effective then this group is in fact a complete invariant.
However, the topological full group of a Cantor monoid is nothing other than the group of units of its associated Tarski monoid
whereas minimal and effective translate exactly into the Tarski monoid being simple (in an appropriate algebraic sense).
It follows that for simple Tarski monoids their groups of units are a complete invariant.
The goal of this paper is to prove this result algebraically.
Section~2 recalls the definitions and Section~3 gives the proof.

\section{Tarski monoids}

The goal of this section is essentially descriptive and will draw upon previously published results due to the author and his collaborators.
The theorem proved in this paper is set out at the end of this section and its proof is given in Section~3.

We begin with the \'etale groupoids studied by Matui \cite{Matui13}.
The best introduction to \'etale (topological) groupoids is \cite{Res1} and we refer the reader there for more information.
Matui is interested in his paper in Hausdorff, second countable, locally compact \'etale topological groupoids but his main results
also assume that the identity space of the groupoid is the Cantor space.
For the purposes of this paper, it is convenient to call such a groupoid a {\em Cantor groupoid}.
The significance of \'etale topological groupoids is that they can be viewed as `non-commutative' topological spaces.
This idea was spelt out by Kumjian \cite{Kumjian} and lies behind Renault's influential approach to constructing $C^{\ast}$-algebras
from such groupoids \cite{Renault}.
Thus Cantor groupoids should be viewed as non-commutative generalizations of the Cantor space.

We now turn to the corresponding non-commutative analogues of the Tarski algebra.
These will turn out to be special kinds of inverse semigroups.
For background in inverse semigroup theory, we refer the reader to \cite{Lawson98}.
If $S$ is a monoid, its group of units is denoted by $\mathsf{U}(S)$; the groups of units of Tarski monoids play an important r\^ole in Matui's work, as we shall see.
If $e$ is an idempotent then $eSe$ is called a {\em local monoid}.
A semigroup $S$ is {\em inverse} if for each $s \in S$ there exists a unique element $s^{-1}$ such that the following two equations hold:
$s = ss^{-1}s$ and $s^{-1} = s^{-1}ss^{-1}$.
We define $\mathbf{d}(s) = s^{-1}s$ and $\mathbf{r}(s) = ss^{-1}$.
This leads to the following diagrammatic way of representing the elements of an inverse semigroup
$$\mathbf{d}(s) \stackrel{s}{\longrightarrow} \mathbf{r}(s).$$
If we restrict the product in the semigroup to those pairs $(s,t)$ where $\mathbf{d}(s) = \mathbf{r}(t)$ we get the {\em restricted product}.
There are a number of important first consequences of the definition of an inverse semigroup, the most important of which \cite[Theorem 1.1.3]{Lawson98}  is that the set of idempotents of $S$,
denoted by $\mathsf{E}(S)$, forms a commutative subsemigroup.\footnote{For the benefit of any ring-theorists reading, it does not follow that the idempotents are therefore central.}
The order used in an inverse semigroup will always be the {\em natural partial order} defined by $a \leq b$ if and only if $a = ba^{-1}a$ if and only if $a = aa^{-1}b$.
If $X \subseteq S$ then $X^{\uparrow} = \{ s\in S \colon x \leq s \mbox{ some } x \in X\}$ and $X^{\downarrow} = \{s \in S \colon s \leq x \mbox{ some } x \in X\}$.
If $X = X^{\downarrow}$ we say that $X$ is an {\em order ideal}.
The {\em compatibility relation} in an inverse monoid is denoted by $\sim$ and defined by $s \sim t$ if and only if $s^{-1}t,st^{-1} \in \mathsf{E}(S)$.
The {\em orthogonality relation} is defined by $s \perp t$ if and only if $s^{-1}t = 0 = st^{-1}$.
A finite subset is said to be {\em compatible} (resp. {\em orthogonal}) if each pair of distinct elements from the set is compatible (resp. orthogonal).

An inverse monoid is said to be {\em Boolean} if its semilattice of idempotents is a Boolean algebra with respect to the natural partial order,
if all binary compatible joins of elements exist and if multiplication distributes over such binary compatible joins.
Boolean inverse monoids are an important class of inverse monoids \cite{KLLR,Lawson10b,Lawson12,Lawson16,LL,LS,Wehrung}.
An inverse monoid is a {\em $\wedge$-monoid} if each pair of elements has a meet.\\

\noindent
{\bf Definition. }A {\em Tarski monoid} is a countably infinite, Boolean inverse $\wedge$-monoid.\\

It is a theorem \cite[Corollary~4.5]{Lawson16} that the Boolean algebra of idempotents of a Tarski monoid is the Tarski algebra.
A monoid homomorphism between Boolean inverse monoids $S$ and $T$ is said to be {\em additive} if it
preserves compatible joins and induces a map of Boolean algebras between $\mathsf{E}(S)$ and $\mathsf{E}(T)$.
The following basic result from Boolean algebra will be used a number of times.

\begin{lemma}\label{lem:basic} 
In a Boolean algebra, we have that $e \leq f$ if and only if $e\bar{f} = 0$.
\end{lemma}

Let $S$ be a Boolean inverse monoid.
If $X \subseteq S$ denote by $X^{\vee}$ the set of all joins of finite compatible subsets of $S$.
A subset $X$ is said to be {\em $\vee$-closed} if $X = X^{\vee}$.
A Boolean inverse monoid is said to be {\em $0$-simplifying} if there are no non-trivial $\vee$-closed ideals.
Let $e$ and $f$ be non-zero idempotents in a Boolean inverse monoid.
Define $e \preceq f$ if there is a set $X = \{x_{1}, \ldots, x_{m} \}$ such that $e = \bigvee_{i=1}^{m} \mathbf{d}(x_{i})$ and $\mathbf{r}(x_{i}) \leq f$
for $1 \leq i \leq m$.
The set $X$ is called a {\em pencil} from $e$ to $f$.
Define $e \equiv f$ if $e \preceq f$ and $f \preceq e$.
The following was proved as \cite[Lemma~7.8]{Lenz}.

\begin{lemma}\label{lem:darwin} Let $S$ be a Boolean inverse monoid.
Then $\equiv$ is the universal relation on the set of non-zero idempotents if and only if $S$ is $0$-simplifying.
\end{lemma}

An inverse monoid is said to be {\em fundamental} if the only elements that centralize the set of idempotents are themselves idempotents.
The following is well-known.

\begin{lemma}\label{lem:pros}   
In a fundamental inverse semigroup all local monoids are fundamental.
\end{lemma}

\noindent
{\bf Definition. }A Boolean inverse monoid is {\em simple} if it is fundamental and $0$-simplifying.\\

It is a theorem \cite[Theorem~4.18]{Lawson12} that the {\em finite} simple  Boolean inverse $\wedge$-monoids are precisely the finite symmetric inverse monoids.
It is for this reason that Tarski monoids are assumed to be countably infinite.
It can be proved that the only additive homomorphisms between simple Boolean inverse monoids are the trivial ones.
It follows that our use of the word `simple' in this context is appropriate.

The Stone space of the Boolean algebra of a Boolean inverse monoid $S$ is called the {\em structure space} of $S$ and is denoted by $\mathsf{X}(S)$.
The elements of the structure space are ultrafilters.
The open sets are of the form $U_{e}$, the set of all ultrafilters containing the idempotent $e$.
We remind the reader that in a Boolean algebra ultrafilters and prime filters are the same \cite{PJ}.

We shall on one occasion need to work with ultrafilters on the Boolean inverse monoid $S$ itself.
The set of such ultrafilters is denoted by $\mathsf{G}(S)$.
It is worth noting that ultrafilters are also prime filters in a suitable sense \cite[Lemma~3.20]{LL}.
These are the basis of non-commutative Stone duality and connect Tarski inverse monoids with Cantor groupoids \cite{Lawson16}.
The set of all ultrafilters containing the element $s$ is $V_{s}$ and, significantly, is a compact set.
Observe that under the assumption of (AC), 
\cite[Lemma~3.11]{LL} implies that in a Boolean inverse monoid $V_{s} \subseteq V_{t}$ if and only if $s \leq t$.
The following result is \cite[Lemma~2.4]{Lawson10b}.

\begin{lemma}\label{lem:exel} Let $A$ be a proper filter in the Boolean inverse $\wedge$-monoid $S$.
Then $A$ is an ultrafilter if and only if $s \wedge t \neq 0$ for all $t \in A$ implies that $s \in A$.
\end{lemma}

If $A$ is an ultrafilter in the Boolean inverse monoid $S$ define $\mathbf{d}(A) = (A^{-1}A)^{\uparrow}$ and $\mathbf{r}(A) = (AA^{-1})^{\uparrow}$.
Define $A \cdot B = (AB)^{\uparrow}$ if and only if $\mathbf{d}(A) = \mathbf{r}(B)$.
With respect to these operations the set $\mathsf{G}(S)$ is a groupoid with its identities being those ultrafilters that contain idempotents.
The papers \cite{Lawson10b, Lawson12, LL} deal with the theory of non-commutative Stone duality in increasing generality and \cite{Lawson16}
refines the theory in certain cases that are relevant to this paper.

The relationship between the structure of a Tarski monoid and the structure of its group of units is the main theme of this paper.
We therefore need ways of constructing units.
An element $g$ of a group is said to be an {\em involution} if $g \neq 1$ and $g^{2} = 1$.
An {\em infinitesimal} in an inverse monoid is an element $a$ such that $a^{2} = 0$.
The {\em extent} of an element $a$, denoted by $\mathbf{e}(a)$, in a Boolean inverse monoid is the idempotent $\mathbf{d}(a) \vee \mathbf{r}(a)$.

\begin{lemma}\label{lem:spooks} Let $S$ be a Boolean inverse monoid.
\begin{enumerate}

\item Let $s \in S$ such that $s^{-1}s = ss^{-1}$.
Put $e = s^{-1}s$.
Then $g = s \vee \bar{e}$ is invertible.

\item  $a^{2} = 0$ if and only if $a^{-1}a \perp aa^{-1}$  if and only if $a \perp a^{-1}$.

\item If $a^{2} = 0$ then
$$g = a^{-1} \vee a \vee \overline{\mathbf{e}(a)}$$
is an involution above $a$.

\end{enumerate}
\end{lemma}
\begin{proof} (1) Simply observe that $g^{-1}g = gg^{-1} = s^{-1}s \vee \overline{s^{-1}s} = 1$.

(2) If $a^{2} = 0$ then $a^{-1}aaa^{-1} = 0$ and so $a^{-1}a \perp aa^{-1}$.
If  $a^{-1}a \perp aa^{-1}$ then $a^{-1}aaa^{-1} = 0$ and so $a^{2} = 0$.
The equivalence of $a^{-1}a \perp aa^{-1}$  with $a \perp a^{-1}$ is immediate.

(3) The elements $a$ and $a^{-1}$ are orthogonal.
Put $s = a \vee a^{-1}$.
Then $s^{-1}s = ss^{-1}$.
Now apply part (1).
It is straightforward to check that it is an involution.
\end{proof}

\noindent
{\bf Definition.} A {\em special involution} is one constructed as in part (3) of Lemma~\ref{lem:spooks}.\\

The following is tangential to the main results of this paper but of independent interest.
An inverse semigroup $S$ is said to be {\em Clifford} if  $\mathbf{d}(s) = \mathbf{r}(s)$ for each element $s \in S$.
It is a classical theorem \cite[Theorem~5.2.12]{Lawson98} that an inverse semigroup is Clifford if and only if its idempotents are central.

\begin{proposition}\label{prop:annoying} Let $S$ be a Boolean inverse semigroup.
The following are equivalent.
\begin{enumerate}
\item $S$ is Clifford.
\item $S$ contains no infinitesimals.
\end{enumerate}
\end{proposition}
\begin{proof} (1)$\Rightarrow$(2). Let $a$ be an infinitesimal.
Then $a^{2} = 0$ and so $a^{-1}aaa^{-1} = 0$.
By assumption, $a^{-1}a = aa^{-1}$.
Thus $a^{-1}a = 0$ from which it follows that $a = 0$, a contradiction. 

(2)$\Rightarrow$(1). Suppose that there were an element $a \in S$ such that $a^{-1}a \neq aa^{-1}$. 
If  $a^{-1}aaa^{-1} = 0$ then $a$ would be an infinitesimal.
Thus $e = \mathbf{d}(a)\mathbf{r}(b) \neq 0$.
Since $e \leq \mathbf{d}(a)$ there is an idempotent $f$ such that $f \leq \mathbf{d}(a)$ and $fe = 0$.
Observe that $af \neq 0$ but then $(af)^{2} = afaf = a\mathbf{d}(a)f\mathbf{r}(a)f = aefa = 0$ and $af$ is an infinitesmial, which is a contradiction.
It follows that no such element $a$ can exist and so $S$ is a Clifford semigroup.
\end{proof}

A pair of infinitesimals $(b,a)$ is called a {\em $2$-infinitesimal} if $\mathbf{d}(b) = \mathbf{r}(a)$ and $ba$ is an infinitesimal.
The {\em extent} of $(b,a)$, denoted by $\mathbf{e}(b,a)$, is $\mathbf{r}(b) \vee \mathbf{d}(b) \vee \mathbf{d}(a)$.
A unit of order $3$ is called a {\em 3-cycle}.

\begin{lemma}\label{lem:julie} 
Let $(b,a)$ be a $2$-infinitesimal in the Boolean inverse monoid $S$ and put $c = (ba)^{-1}$ as in the following diagram
 
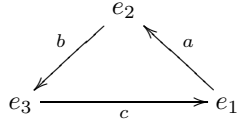
\begin{equation*}
\xymatrix{
& e_{2} \ar[dl]_{b} & \\
e_{3} \ar[rr]_{c}  & & e_{1} \ar[ul]_{a}  }
\end{equation*}
where the idempotents $e_{1},e_{2},e_{3}$ are mutually orthogonal.
Put $e = e_{1} \vee e_{2} \vee e_{3}$.
\begin{enumerate}
\item $\{a,b,c\}$ is an orthogonal set and $a \vee b \vee c$ is an element of $\mathsf{U}(eSe)$ of order $3$.

\item $g = a \vee b \vee c \vee \bar{e}$ is a $3$-cycle.

\item Put $h = a \vee a^{-1} \vee \overline{\mathbf{e}(a)}$ and $k = c \vee c^{-1} \vee \overline{\mathbf{e}(c)}$.
Then $h$ and $k$ are special involutions and $g = [h,k] = (hk)^{2}$.

\end{enumerate}
\end{lemma}
\begin{proof} (1) It is easy to check that  $\{a,b,c\}$ is an orthogonal subset and that $g = a \vee b \vee c$
is a unit in $eSe$.
Observe that $(a \vee b \vee c)^{2} = a^{-1} \vee b^{-1} \vee c^{-1} = (a \vee b \vee c)^{-1}$.
Thus  $(a \vee b \vee c)^{3} = e$.

(2) This is immediate by part (1).

(3) Routine verification.
 \end{proof}

\noindent
{\bf Definition. }A {\em special $3$-cycle} is one constructed as in part (2) of Lemma~\ref{lem:julie}. 

\begin{remark}\label{rem:whisky}
{\em By part (3) of Lemma~\ref{lem:julie}, each special $3$-cycle is a product of two special involutions.}
\end{remark}

\begin{example} {\em Let $I_{n}$ be a finite symmetric inverse monoid on $n$ letters.
Examples of infinitesimal elements are those elements of the form $x \mapsto y$ where $x,y \in X$ and $x \neq y$.
The group elements associated with these, as constructed in the above lemma, are precisely the transpositions.
The group generated by the involutions constructed from infinitesimals is precisely the symmetric group on $n$ letters,
the group of units of $I_{n}$.}
\end{example}

The above example motivates the following.\\

\noindent
{\bf Definition. }Let $S$ be a Tarski monoid.
Define $\mathsf{Sym}(S)$ to be the subgroup of $\mathsf{U}(S)$ generated by the special involutions 
and define $\mathsf{Alt}(S)$ to be the subgroup of $\mathsf{U}(S)$ generated by the special $3$-cycles.
Clearly, $\mathsf{Alt}(S), \mathsf{Sym}(S) \trianglelefteq \mathsf{U}(S)$ and  $\mathsf{Alt}(S) \leq \mathsf{Sym}(S)$ by Remark~\ref{rem:whisky}.
We call $\mathsf{Sym}(S)$ the {\em symmetric group} of $S$ and $\mathsf{Alt}(S)$ the {\em alternating group} of $S$.\\

These subgroups are the subject of \cite{Nek} defined there in a different, but equivalent, way.

Matui's spatial realization theorem \cite[Theorem 3.10]{Matui13} is phrased entirely in the language of \'etale groupoids.
By applying non-commutative Stone duality \cite{Lawson16}, the equivalence of (1) and (2) below is immediate
and that of (1) and (3) follows from \cite{Nek} (who appeals to to \cite{Rubin}).

\begin{theorem}[Matui's spatial realization theorem]\label{them:matui} Let $S$ and $T$ be two simple Tarski monoids.
Then the following are equivalent.
\begin{enumerate}
\item The monoids $S$ and $T$ are isomorphic.
\item The groups $\mathsf{U}(S)$ and $\mathsf{U}(T)$ are isomorphic.
\item  The symmetric groups $\mathsf{Sym}(S)$ and $\mathsf{Sym}(T)$ are isomorphic.
\end{enumerate}
\end{theorem}

The inverse monoid approach described in our paper is entirely consonant with the approach adopted by Matui. He works throughout with what he terms
`compact open $G$-sets' which  are precisely the `compact-open local bisections' needed in non-commutative Stone duality \cite{Lawson10b,Lawson12,Lawson12,Lawson16,LL}. 
Thus many of Matui's calculations are actually being carried out in the
Tarski monoid associated with his groupoid.

There is a finitary version of the above theorem that is worth stating. The finite simple Boolean inverse monoids are precisely the finite symmetric inverse monoids $I_{n}$
and the groups of units of such monoids are exactly the finite symmetric groups $S_{n}$.
Clearly, $S_{m} \cong S_{n}$ if and only if $m = n$ if and only if $I_{m} \cong I_{n}$.
In particular, this suggests that the groups of units of simple Tarski monoids should be regarded as infinite generalizations of finite symmetric groups.
This is not a novel idea but the inverse monoid context makes the point very clearly.

Finally, there are two interesting families of Tarski monoids that are worth highlighting.
The {\em Cuntz inverse monoids} $C_{n}$, where $n \geq n$,  are described in \cite{Lawson07b, LS}.
Their groups of units are the Thompson groups $V_{n}$.
The {\em AF inverse monoids} are described in \cite{LS} and include a class of (simple) Tarski inverse monoids.
In particular, the groups studied in \cite{KS, LN} arise as groups of units of simple AF inverse monoids which are Tarski monoids.

\section{An algebraic proof of the spatial realization theorem}

The goal of this section is to prove Theorem~\ref{them:matui} directly.
We shall, of course, reinterpret many ideas due to Matui but our approach opens the door to obtaining an algebraic characterization\footnote{In \cite{Lawson16}, a characterization is obtained of these groups
but it assumes that the group is already given as a subgroup of the group of homeomorphisms of the Cantor space.} of the groups of units of simple Tarski monoids.

\subsection{Preliminaries}

Let $S$ be a Boolean inverse monoid.
Then the group of units of $S$ acts on the set of idempotents of $S$
by $e \mapsto geg^{-1}$, where $e \in \mathsf{E}(S)$ and $g \in \mathsf{U}(S)$.
We call this the {\em natural action} on the Boolean algebra of idempotents.
The following is an application of part (1) of Lemma~\ref{lem:spooks}.

\begin{proposition}\label{prop:fc} Let $S$ be a Boolean inverse monoid.
Then $S$ is fundamental if and only if the natural action of $\mathsf{U}(S)$ on $\mathsf{E}(S)$ is faithful.
\end{proposition}
\begin{proof} Suppose first that $S$ is fundamental.
Let $g$ be a unit such that $geg^{-1} = e$ for all idempotents $e$.
Then $ge = eg$ for all idempotents $e$.
By assumption, $g$ is an idempotent and so the group identity.
Thus the natural action is faithful.
Conversely, suppose that the natural action is faithful.
Let $a \in S$ be such that $ae = ea$ for all idempotents $e$.
In particular, $a = a(a^{-1}a) = (a^{-1}a)a$.
Thus $aa^{-1} = (a^{-1}a)aa^{-1}$ giving 
$aa^{-1} \leq a^{-1}a$ which
by symmetry yields $a^{-1}a = aa^{-1}$.
Put $g = a \vee \overline{a^{-1}a}$, a unit by Lemma~\ref{lem:spooks}.
Observe that $ge = eg$ for all idempotents $e$.
But the natural action is faithful.
Thus $g$ is the identity and so $a$ is an idempotent
from which it follows that $S$ is fundamental.
\end{proof}

There is an useful consequence of the above result.

\begin{lemma}\label{lem:tea-time} Let $S$ be a fundamental Boolean inverse monoid.
\begin{enumerate}

\item Let $e$ be an idempotent and $g$ a unit such that $geg^{-1} \neq e$.
Then there is a non-zero idempotent $f$ such that $f \leq e$ and $f \perp gfg^{-1}$. 

\item For each non-trivial unit $g$ there is a non-zero idempotent $e$ such that $e \perp geg^{-1}$.

\end{enumerate}
\end{lemma} 
\begin{proof} (1) By \cite[Proposition~I.2.5]{PJ},
and using the fact that in a Boolean algebra ultrafilters are prime filters,
there are two possibilities.
There is an ultrafilter $F \subseteq \mathsf{E}(S)$ such that
either 
($e \in F$ and $geg^{-1} \notin F$)
or
($e \notin F$ and $geg^{-1} \in F$).
In the first case, this means by Lemma~\ref{lem:exel} that there is $f \in F$, which can be chosen so that $f \leq e$, such that $f(geg^{-1}) = 0$.
It follows that $f(gfg^{-1}) = 0$ since $gfg^{-1} \leq geg^{-1}$.
In the second case, observe that $F' = g^{-1}Fg$ is an ultrafilter
and that $e \in F'$ and $g^{-1}eg \notin F'$.
Thus there is an $f \in F'$ such that $f(g^{-1}fg) = 0$.
It follows that $(gfg^{-1})f = 0$.

(2) Let $g$ be a non-trivial unit. 
By Proposition~\ref{prop:fc},
there is an idempotent $f$ such that $f \neq gfg^{-1}$.
We now apply part (1) above.
\end{proof}

Let $S$ be a Boolean inverse monoid.
Then the group of units of $S$ acts on structure space of $S$ 
by $F \mapsto gFg^{-1}$, where $F \in \mathsf{X}(S)$ and $g \in \mathsf{U}(S)$.
We call this the {\em natural action} on the structure space.

\begin{lemma}\label{lem:soros} Let $S$ be a Boolean inverse monoid.
\begin{enumerate}

\item  Let $F$ be an ultrafilter in $\mathsf{E}(S)$ and let $g$ be a unit.
Then $F \neq gFg^{-1}$ if and only if there exists $e \in F$ such that $e \perp geg^{-1}$.

\item Let $F$ be an ultrafilter in $\mathsf{E}(S)$ and let $g,h$ be units such that $F, gFg^{-1}, hFh^{-1}$ are distinct.
Then there exists an $e \in F$ such that $\{ e,geg^{-1},heh^{-1} \}$ is an orthogonal set.

\end{enumerate}
\end{lemma}
\begin{proof} (1) Only one direction needs proving.
Suppose that $F \neq gFg^{-1}$.
Then there exists $gig^{-1} \in gFg^{-1}$, where $i \in F$, such that $gig^{-1} \notin F$.
It follows by Lemma~\ref{lem:exel} that there exists $j \in F$ such that $j \perp gig^{-1}$.
Put $e = ij \in F$.
Then $e \perp geg^{-1}$.

(2) We use part (1) repeatedly. 
Since $F \neq gFg^{-1}$ there exists $i \in F$ such that $i \perp gig^{-1}$.
Since $F \neq hFh^{-1}$ there exists $j \in F$ such that $j \perp hjh^{-1}$.
Since $gFg^{-1} \neq hFh^{-1}$ there exists $k \in F$ such that $gkg^{-1} \perp hkh^{-1}$.
Put $e = ijk \in F$.
Then $\{e, geg^{-1}, heh^{-1}\}$ is an orthogonal set. 
\end{proof}

The following lemma is significant because it connects the natural actions with the natural partial order. 

\begin{lemma}\label{lem:blue} Let $S$ be a fundamental Boolean inverse monoid.
Let $g$ be a unit and $e$ an idempotent.
Then the following are equivalent.
\begin{enumerate}
\item $g$ fixes the set $U_{e}$ pointwise under the natural action.
\item $g$ fixes the set $e^{\downarrow}$ pointwise under the natural action.
\item $e \leq g$.
\end{enumerate}
\end{lemma}
\begin{proof} 
(1)$\Rightarrow$(2). Suppose that  $g$ fixes the set $U_{e}$ pointwise under the natural action.
Then it is immediate that $g^{-1}$ fixes the set $U_{e}$ pointwise under the natural action as well.
Let $0 \neq f \leq e$ and suppose that $gfg^{-1} \neq f$.
There are two cases.
Suppose first that there is an ultrafilter $F$ that contains $f$ and does not contain $gfg^{-1}$.
Then $gFg^{-1} \neq F$.
But $f \in F$ implies that $e \in F$ and we get a contradiction.
Suppose now that there is an ultrafilter $G$ that contains $gfg^{-1}$ and does not contain $f$.
We now apply the same argument as above except with $g^{-1}$ instead of $g$ and get another contradiction.
It follows that $g$ fixes the set $e^{\downarrow}$ pointwise under the natural action.

(2)$\Rightarrow$(3). Suppose that $g$ fixes the set $e^{\downarrow}$ pointwise under the natural action.
Then, in particular, $\mathbf{d}(ge) = \mathbf{r}(ge) = e$.
Thus $ge$ is a unit in the local monoid $eSe$.
By assumption, $ge$ commutes with every idempotent in $eSe$.
But $S$ fundamental implies that $eSe$ is fundamental by Lemma~\ref{lem:pros}.
Thus $ge$ is an idempotent and so equals $e$.
Thus $e = ge \leq g$, as required. 

(3)$\Rightarrow$(1).
Suppose that  $e \leq g$.
Let $F \in U_{e}$ and $f \in F$.
Then $ef \leq gfg^{-1}$.
But $ef \in F$.
Thus $gfg^{-1} \in F$.
It follows that $gFg^{-1} \subseteq F$.
But $gFg^{-1}$ is itself an ultrafilter.
Thus $gFg^{-1} = F$, as required.
\end{proof}

So far, we have left open the question of the existence of infinitesimals.
It is here that we shall need to assume $0$-simplifying.

\begin{lemma}\label{lem:george} Let $S$ be a $0$-simplifying Tarski monoid.
Let $F \subseteq \mathsf{E}(S)$ be an ultrafilter and let $e \in F$.
Then there exists an element $a \in S$ such that
\begin{enumerate}

\item $a$ is an infinitesimal.

\item $a^{-1}a \in F$.

\item $a \in eSe$.

\end{enumerate}
\end{lemma}
\begin{proof}
The idempotent $e \neq 0$.
We are working in a Tarski algebra, and so $e$ cannot be an atom.
Thus there exists $0 \neq f < e$.
The idempotents form a Boolean algebra, and so $e = f \vee \bar{f}$ and $f \wedge \bar{f} = 0$.
Since $f \vee \bar{f} = e \in F$, and $F$ is an ultrafilter and so a prime filter, we know that $f \in F$ or $\bar{f} \in F$.
Without loss of generality, we may assume that $f \in F$.
Now $S$ is $0$-simplifying and so $\bar{f} \equiv f$.
In particular, $f \preceq \bar{f}$.
We may therefore find elements $x_{1}, \ldots, x_{m}$ 
such that
$f = \bigvee_{i=1}^{m} \mathbf{d}(x_{i})$ 
and $\mathbf{r}(x_{i}) \leq \bar{f}$.
We use the fact that $F$ is a prime filter, to deduce that $\mathbf{d}(x_{i}) \in F$ for some $i$.
Put $a = x_{i}$.
Then $a^{-1}a \leq f$ and $aa^{-1} \leq \bar{f}$.
Hence $a^{-1}a \perp aa^{-1}$.
It follows that $a$ is an infinitesimal.
Clearly,  $a^{-1}a,aa^{-1} \leq e$.
In addition, $a^{-1}a \in F$.
\end{proof}

The above result implies the existence of of $2$-infinitesimals, as we now show.

\begin{lemma}\label{lem:horsa} Let $S$ be a $0$-simplifying Tarski monoid. 
Let $e$ be any non-zero idempotent.
Then there exist infinitesimals $a,b \in eSe$ such that $(b,a)$ is a $2$-infinitesimal.
\end{lemma}
\begin{proof} Every non-zero idempotent is an element of some ultrafilter in $E(S)$.
Thus by Lemma~\ref{lem:george}, we may find an infinitesimal $x \in eSe$.
Similarly, we may find an infinitesimal $a \in \mathbf{d}(x)S\mathbf{d}(x)$.
Put $b = x \mathbf{r}(a)$.
The set of infinitesimals forms an order ideal, and so $b$ is an infinitesimal.
By construction, $ba$ is a restricted product and  $\mathbf{r}(b) \perp \mathbf{d}(a)$. 
\end{proof}

\subsection{Ingredients for an isomorphism}

In this section, we will lay the foundations of the proof of our main theorem
in that  we shall describe what ingredients are needed to construct an isomorphism.

\begin{lemma}\label{lem:ugly} Let $S_{1}$ and $S_{2}$ be fundamental Boolean inverse monoids.
Let $G_{1} \leq \mathsf{U}(S_{1})$ and $G_{2} \leq \mathsf{U}(S_{2})$ be subgroups and suppose that
$\alpha \colon G_{1} \rightarrow G_{2}$ is an isomorphism of groups,
$\gamma \colon \mathsf{E}(S_{1}) \rightarrow \mathsf{E}(S_{2})$ is an isomorphism of Boolean algebras
and
$$\gamma (geg^{-1}) = \alpha (g) \gamma (e) \alpha (g)^{-1}$$
for all $g \in G_{1}$ and $e \in \mathsf{E}(S_{1})$.
Then $e \leq g$ if and only if $\gamma (e) \leq \alpha (g)$
for all $g \in G_{1}$ and $e \in \mathsf{E}(S_{1})$.
\end{lemma}
\begin{proof} Only one direction needs proving.
Suppose that $e \leq g$.
Thus $g$ fixes the set $U_{e}$ pointwise under conjugation by Lemma~\ref{lem:blue}.
Let $F \in U_{\gamma (e)}$ and let $f \in F$.
Then $\gamma^{-1}(f) \in \gamma^{-1}(F)$ which is an ultrafilter containing $e$.
Thus $g \gamma^{-1}(f)g^{-1} \in F$.
But $\gamma (g \gamma^{-1} (f) g^{-1}) = \alpha (g) f \alpha (g)^{-1}$.
It follows that $F$ is mapped to itself under the natural action by $\alpha (g)$
and so $\alpha (g)$ fixes the set $U_{\gamma (e)}$ pointwise.
Hence $\gamma (e) \leq \alpha (g)$ by  Lemma~\ref{lem:blue}.
\end{proof}

\begin{lemma}\label{lem:trump}
 Let $S_{1}$ and $S_{2}$ be Boolean inverse monoids
and let $G_{1} \leq \mathsf{U}(S_{1})$ and $S_{2} \leq \mathsf{U}(S_{2})$ be subgroups such that $S_{1} = (G_{1}^{\downarrow})^{\vee}$ and $S_{2} = (G_{2}^{\downarrow})^{\vee}$.
Let $\theta \colon G_{1}^{\downarrow} \rightarrow G_{2}^{\downarrow}$ be an isomorphism.
Then $\theta$ extends uniquely to an isomorphism $\Theta \colon S_{1} \rightarrow S_{2}$.
\end{lemma}
\begin{proof} Let $\phi \colon S_{1} \rightarrow S_{2}$ be any additive homomorphism that extends $\theta$.
Let $s \in S_{1}$.
Then we can write $s = \bigvee_{i=1}^{m} g_{i}e_{i}$ where $g_{i} \in G_{1}$ and $e_{i} \in \mathsf{E}(S_{1})$ by assumption.
Since $\phi$ is additive we have that $\phi (s) =  \bigvee_{i=1}^{m} \phi(g_{i}e_{i}) =   \bigvee_{i=1}^{m} \theta (g_{i}e_{i})$.
This proves that if $\theta$ can be extended then that extension is unique.
It remains to prove that $\theta$ can be extended.

Let $s,t \in (G_{1})^{\downarrow}$.
Then $s = ge$ and $t = hf$ where $g,h \in G_{1}$ and $e,f \in \mathsf{E}(S)$.
Since $\theta$ is a homomorphism $s \sim t$ implies that $\theta (s) \sim \theta (t)$.

Let $s \in S_{1}$ be arbitrary.
Then we can write $s = \bigvee_{i=1}^{m} g_{i}e_{i}$ where $g_{i} \in G_{1}$ and $e_{i} \in \mathsf{E}(S_{1})$ by assumption.
Observe that $g_{i}e_{i} \sim g_{j}e_{j}$ and so $\theta (g_{i}e_{i}) \sim \theta (g_{j}e_{j})$.
It follows that we may define
$$\Theta (s) = \bigvee_{i=1}^{m} \theta (g_{i}e_{i}).$$
However, we need to check that this is independent of the way we described $s$.
Suppose that $s = \bigvee_{j=1}^{n} h_{j}f_{j}$ where $h_{j} \in G_{1}$ and $f_{j} \in \mathsf{E}(S)$.
We need to prove that
$$ \bigvee_{i=1}^{m} \theta (g_{i}e_{i})
=
\bigvee_{j=1}^{n} \theta (h_{j}f_{j}).$$
We have that $g_{i}e_{i} \leq \bigvee_{j=1}^{n} h_{j}f_{j}$.
We prove that $\theta (g_{i}e_{i}) \leq \bigvee_{j=1}^{n} \theta (h_{j}f_{j})$.
Observe first that since $g_{i}e_{i} \sim h_{j}f_{j}$ it follows that
$$g_{i}e_{i} \wedge h_{j}f_{j} = g_{i}e_{i}f_{j} = h_{j}e_{i}f_{j}.$$
It follows that
$$g_{i}e_{i} = \bigvee_{j=1}^{n} g_{i}e_{i}f_{j}$$
by part (3) of \cite[Lemma~2.5]{Lawson16}.
Now this join takes place entirely within $G_{1}^{\downarrow}$ and we have proved that $\theta$
is an isomorphism.
Thus 
$$\theta (g_{i}e_{i}) = \bigvee_{j=1}^{n} \theta (g_{i}e_{i}f_{j}).$$
But $g_{i}e_{i}f_{j} \leq g_{i}e_{i}$.
It is now immediate that
$$\theta (g_{i}e_{i}) \leq \bigvee_{j=1}^{n} \theta (h_{j}f_{j}).$$
This now implies that $\Theta$ is well-defined.
The fact that $\Theta$ is a homomorphism follows from the fact that $\theta$ is a homomorphism when restricted to $G_{1}^{\downarrow}$.
We have therefore constructed a unique homomorphism $\Theta \colon S_{1} \rightarrow S_{2}$  that extends $\theta$.
By symmetry, we may also construct a unique homomorphism $\Phi \colon S_{2} \rightarrow S_{1}$  that extends $\theta^{-1}$.
Observe that $(\Theta \mid G_{1}^{\downarrow})$ and $(\Phi \mid G_{2})$ are inverses of each other from which it readily follows that $\Phi = \Theta^{-1}$.
We have therefore proved that $\Theta$ is an isomorphism. 
\end{proof}

\begin{lemma}\label{lem:lino} Let $S_{1}$ and $S_{2}$ be Boolean inverse monoids
and let $G_{1} \leq \mathsf{U}(S_{1})$ and $S_{2} \leq \mathsf{U}(S_{2})$ be subgroups such that $S_{1} = (G_{1}^{\downarrow})^{\vee}$ and $S_{2} = (G_{2}^{\downarrow})^{\vee}$.
Suppose that there is an isomorphism $\alpha \colon G_{1} \rightarrow G_{2}$ of groups
and an isomorphism $\gamma \colon \mathsf{E}(S_{1}) \rightarrow \mathsf{E}(S_{2})$ of Boolean algebras 
such that the following two properties hold:
\begin{enumerate}
\item For all $g \in G_{1}$ and $e \in \mathsf{E}(S_{1})$ we have that
$$\gamma (geg^{-1}) = \alpha (g) \gamma (e) \alpha (g)^{-1}.$$
\item Let $e \in \mathsf{E}(S)$ and $g \in G_{1}$.
Then  $e \leq g$ if and only if $\gamma (e) \leq \alpha (g)$.
\end{enumerate}
Then there is an isomorphism $\Theta \colon S_{1} \rightarrow S_{2}$ that extends both $\alpha$ and $\gamma$.
\end{lemma}
\begin{proof} We begin by considering the elements of $S_{1}$ in the inverse submonoid $G_{1}^{\downarrow} = G_{1} \mathsf{E}(S_{1})$.
Let $g,h \in G_{1}$ and $e,f \in \mathsf{E}(S_{1})$.
If $ge = hf$ then in fact $e = f$ by calculating domains.
Thus $ge = he$.
It follows that $e = (g^{-1}h)e$ and so $e \leq g^{-1}h$.
By property (2), we have that $\gamma (e) \leq \alpha (g^{-1}h)$.
It follows that $\alpha (g) \gamma (e) = \alpha (h) \gamma (e)$.
Define $\theta \colon G_{1}^{\downarrow}  \rightarrow G_{2}^{\downarrow}$
by $\theta (ge) = \alpha (g) \gamma (e)$.
This is well-defined by the above argument.
We prove that $\theta$ as defined is a homomorphism.
Let $s = ge$ and $t = hf$.
Then $st = gehf = (gh) h^{-1}ehf$.
By definition $\theta (s) = \alpha (g) \gamma (e)$ and $\theta (t) = \alpha (h) \gamma (f)$
and also
$\theta (st) = \alpha (gh) \gamma (h^{-1}ehf)$.
By property (1) and the fact that $\gamma$ is a homomorphism we have that
$$\alpha (gh) \gamma (h^{-1}ehf)
= \alpha (gh) \alpha (h^{-1}) \gamma (e) \gamma (h) \gamma (f)
=
\alpha (g) \gamma (e) \alpha (h) \gamma (f)
=
\theta (s)\theta (t).
$$
We next prove that $\theta$ is injective.
Suppose that $\theta (ge) = \theta (hf)$.
Then $\alpha (g) \gamma (e) = \alpha (h) \gamma (f)$.
As above, we have that $\gamma (e) = \gamma (f)$ and so, since $\gamma$ is an isomorphism,
we must have that $e = f$.
It follows that $\gamma (e) \leq \alpha (g^{-1}h)$.
Thus by property (2), we have that $e \leq g^{-1}h$ and so $ge = he$.
It follows that $\theta$ is an isomorphism from $G_{1}^{\downarrow}$ to $G_{2}^{\downarrow}$.
In fact, it is worth noting that the properties (1) and (2) exactly characterize isomorphisms 
from $G_{1}^{\downarrow}$ to $G_{2}^{\downarrow}$.
We now use Lemma~\ref{lem:trump} to extend $\theta$ to the required isomorphism.
\end{proof}

The main result of this section is the following.
It follows immediately from Lemma~\ref{lem:ugly} and Lemma~\ref{lem:lino}.

\begin{proposition}\label{prop:isomorphism}
Let $S_{1}$ and $S_{2}$ be fundamental Boolean inverse monoids
and let $G_{1} \leq \mathsf{U}(S_{1})$ and $S_{2} \leq \mathsf{U}(S_{2})$ be subgroups such that $S_{1} = (G_{1}^{\downarrow})^{\vee}$ and $S_{2} = (G_{2}^{\downarrow})^{\vee}$.
Suppose that there is an isomorphism $\alpha \colon G_{1} \rightarrow G_{2}$ of groups
and an isomorphism $\gamma \colon \mathsf{E}(S_{1}) \rightarrow \mathsf{E}(S_{2})$ of Boolean algebras 
such that for all $g \in G_{1}$ and $e \in \mathsf{E}(S_{1})$ we have that
$$\gamma (geg^{-1}) = \alpha (g) \gamma (e) \alpha (g)^{-1}.$$
Then there is an isomorphism $\Theta \colon S_{1} \rightarrow S_{2}$ that extends both $\alpha$ and $\gamma$.
\end{proposition}

In the next section, we shall investigate the requirement that  $S = (G^{\downarrow})^{\vee}$ 
where $G$ is a subgroup of the group of units of $S$.

\subsection{Piecewise factorizability}

An inverse monoid is said to be {\em factorizable} if every element lies beneath an element of the group of units.
For example, symmetric inverse monoids are factorizable if and only if they are finite.
In this paper, we shall need a weaker notion.
Let $G$ be a subgroup of the group of units of the Boolean inverse monoid $S$.
We say that $S$ is {\em piecewise factorizable with respect to $G$} if each element $s \in S$ may 
be written in the form $s = \bigvee_{i=1}^{m} s_{i}$ where for each $s_{i}$ there is a unit $g_{i} \in G$ such that $s_{i} \leq g_{i}$.
This may be rewritten in the following form:
$$s = \bigvee_{i=1}^{m} g_{i}e_{i} $$
where $e_{i} = \mathbf{d}(s_{i})$.
If $G = \mathsf{U}(S)$ then we simply say {\em piecewise factorizable}.
In this section, we prove the following which already hints at the close connection between the structure of the Tarski monoid as a whole and its group of units.

\begin{proposition}\label{prop:dory}  Let $S$ be a simple Tarski monoid.
Then $S$ is piecewise factorizable with respect to $\mathsf{Sym}(S)$.
Thus $S = (\mathsf{Sym}(S)^{\downarrow})^{\vee}$.
\end{proposition}

The proof of this result follows by Lemma~\ref{lem:welsh} and Lemma~\ref{lem:hengist} below.

\begin{lemma}\label{lem:welsh} Let $S$ be a Boolean inverse $\wedge$-monoid and let $G$ be a subgroup of $\mathsf{U}(S)$.
Then $S$ is piecewise factorizable with respect to $G$ if and only if each ultrafilter of $S$ contains an element of $G$.
\end{lemma}
\begin{proof} Suppose first that $S$ is piecewise factorizable.
Let $A$ be any ultrafilter and choose $s \in A$.
Then by assumption we may write 
$s = \bigvee_{i=1}^{m} s_{i}$ where for each $s_{i}$ there is a unit $g_{i} \in G$ such that $s_{i} \leq g_{i}$.
But every ultrafilter is prime and so $s_{i} \in A$ for some $i$.
It is now immediate that $g_{i} \in A$ and so each ultrafilter contains a unit from $G$.

To prove the converse, assume that every ultrafilter contains a unit from $G$.
Let $s \in S$ be any non-zero element.
We shall write $V_{s}$ as a union of compact-open sets.
Let $A \in V_{s}$.
Then there is some unit $g \in A \cap G$.
Thus $g \wedge s \in A$.
We may therefore write
$V_{s} = \bigcup V_{s_{i}}$ where the $s_{i}$ are those elements belonging to the elements of $V_{s}$ which are  beneath units in $G$.
By compactness, we may write $V_{s} = \bigcup_{i=1}^{m} V_{s_{i}}$.
It follows that $s = \bigvee_{i=1}^{m} s_{i}$ where each $s_{i}$ lies beneath a unit in $G$. 
\end{proof}

The following is the key to the proof of Proposition~\ref{prop:dory}.

\begin{lemma}\label{lem:hengist} Let $S$ be a $0$-simplifying Tarski monoid.
\begin{enumerate}
\item  Every ultrafilter contains an infinitesimal or the product of two infinitesimals.
\item Every ultrafilter contains a unit from $\mathsf{Sym}(S)$.
\end{enumerate}
\end{lemma}
\begin{proof} (1) We may restrict our attention to non-idempotent ultrafilters $A$.
Suppose first that  $A$ is an ultrafilter such that $A^{-1} \cdot A \neq A \cdot A^{-1}$.
Both  $A^{-1} \cdot A$ and $A \cdot A^{-1}$ are idempotent ultrafilters and are distinct by assumption.
Since the groupoid $\mathsf{G}(S)$ is Hausdorff there are compact-open sets $V_{s}$ and $V_{t}$ such that
$A^{-1} \cdot A \in V_{s}$ and $A \cdot A^{-1} \in V_{t}$ where $s \wedge t = 0$.
We may find idempotents $e$ and $f$ such that $A^{-1} \cdot A \in V_{e}$ and $A \cdot A^{-1} \in V_{f}$  and $e \wedge f = 0$.
Let $a \in A$ and put $b = fae$.
Then $b \in A$ and $b^{2} = 0$.

We now consider the case where $A$ is an ultrafilter such that $A^{-1} \cdot A = A \cdot A^{-1} = F^{\uparrow}$ where $F \subseteq E(S)$ is an ultrafilter.
We shall prove that there is an ultrafilter $G \subseteq \mathsf{E}(S)$ distinct from $F$ and an ultrafilter $B$ such that $F^{\uparrow} = B \cdot B^{-1}$ and $B^{-1} \cdot B = G^{\uparrow}$.
Then $B^{-1} \cdot B \neq B \cdot B^{-1}$ and $A = (A \cdot B) \cdot B^{-1}$.
By  the first case above, both $A \cdot B$ and $B^{-1}$ contain infinitesimals and so $A$ contains a product of infinitesimals.

Let $F \subseteq \mathsf{E}(S)$ be an ultrafilter.
Let $e \in F$.
Using the fact that $\mathsf{E}(S)$ is a Tarski algebra, we may write $e = e_{1} \vee e_{2}$ where $e_{1},e_{2} \neq 0$ and $e_{1} \perp e_{2}$.
Without loss of generality, we may assume that $e_{1} \in F$ and $e_{2} \notin F$.
We now relabel.
Let $e \in F$ and let $f \neq 0$ be such that $e \perp f$.
By assumption, $e \preceq f$.
Thus there are elements $x_{1}, \ldots, x_{m}$ such that
$e = \bigvee_{i=1}^{m} \mathbf{d}(x_{i})$ and $\mathbf{r}(x_{i}) \leq f$.
Since $F$ is an ultrafilter it is also a prime filter and so, relabelling if necessary, $\mathbf{d}(x_{1}) \in F$.
Consider the ultrafilter $C = (x_{1}F^{\uparrow})^{\uparrow}$.
Then $\mathbf{d}(C) = F^{\uparrow}$.
Put $G = E(\mathbf{r}(C))$.
Then $f \in G$.
It follows that $C \cdot C^{-1} \neq F^{\uparrow}$.

(2) This follows by part (1) and Lemma~\ref{lem:spooks}.
\end{proof}

\subsection{The support operator}

The main tool needed to work with inverse $\wedge$-monoids is the following concept introduced by Leech \cite{Leech}.
Let $S$ be an inverse monoid.
A function $\phi \colon S \rightarrow \mathsf{E}(S)$ is called a {\em fixed-point operator} if it satisfies the following two conditions:
\begin{description}

\item[{\rm (FPO1)}]  $s \geq \phi (s)$.

\item[{\rm (FPO2)}]  If $s \geq e$ where $e$ is any idempotent then $\phi (s) \geq e$.

\end{description}
We have included the proofs below for completeness.

\begin{lemma}\label{lem:fpo} Let $S$ be an inverse monoid.
\begin{enumerate}
\item $S$ has all binary meets if and only if it has a fixed-point operator.
\item In an inverse $\wedge$-monoid the fixed-point operator $\phi$ exists  and is unique being given by $\phi (a) = a \wedge 1$.
\item The fixed-point operator $\phi \colon S \rightarrow \mathsf{E}(S)$ is an order preserving idempotent function having $\mathsf{E}(S)$ as its fixed-point set
that satisfies both $\phi (ae) = \phi(a)e$ and $\phi(ea) = e \phi (a)$ for all $e \in \mathsf{E}(S)$ and all $a \in S$.
\item If $a$ is an infinitesimal then $\phi (a) = 0$.
\item If $a \sim b$ then $\phi (a \vee b) = \phi (a) \vee \phi (b)$.
\item $\phi (a \wedge b) = \phi (a) \wedge \phi (b)$. 
\end{enumerate}
\end{lemma}
\begin{proof} (1) 
Suppose first that $S$ has all binary meets.
For each $a \in S$ define $\phi (a) = a \wedge 1$.
Clearly, $\phi (a)$ is an idempotent.
Let $e$ be an idempotent such that $e \leq a$.
Then $e \leq1$ and so $e \leq \phi (a)$.
Conversely, suppose that a function $\phi$ exists.
Let $a,b \in S$ and consider the element $\phi (ab^{-1})b$.
Clearly, $\phi (ab^{-1})b \leq b$.
But by definition $\phi (ab^{-1}) \leq ab^{-1}$ and so  $\phi (ab^{-1})b \leq ab^{-1}b \leq a$.
Let $c \leq a,b$.
Then $cc^{-1} \leq ab^{-1}$ and so $cc^{-1} \leq \phi (ab^{-1})$.
Now $c \leq b$ and so $c = (cc^{-1})c \leq \phi (ab^{-1})b$.
It follows that $a \wedge b =  \phi (ab^{-1})b$.

(2) This is immediate by (1).

(3) We prove that $\phi (ae) = \phi(a)e$ for all $e \in \mathsf{E}(S)$.
By definition $\phi (ae) \leq ae$.
It also follows from this that $\phi (ae)e = \phi (ae)$.
Thus $\phi (ae) \leq a$ and so $\phi (ae) \leq ae$.
Now $\phi(a) \leq a$ and so $\phi (a)e \leq ae$.
By definition $\phi (a)e \leq \phi (ae)$.
We have therefore proved that $\phi (ae) = \phi (a)e$. 

(4) Let $e$ be an idempotent and $a$ an infinitesimal such that $e \leq a$.
Then $e \leq a^{2} = 0$ and so $e = 0$.

(5) By part (2), we have that $\phi (a \vee b) = (a \vee b) \wedge 1$.
But by a version of the distributivity law that holds for Boolean inverse $\wedge$-monoids,  \cite[Proposition 2.5(3)]{Lawson16},
we have that $(a \vee b) \wedge 1 = (a \wedge 1) \vee (b \wedge 1) = \phi (a) \vee \phi (b)$, as required.

(6) By part (2), we have that $\phi (a \wedge b) = (a \wedge b) \wedge 1 = (a \wedge 1) \wedge (b \wedge 1) = \phi (a) \wedge \phi (b)$.
\end{proof}

\noindent
{\bf Definition. }Let $S$ be a Boolean inverse $\wedge$-monoid.
Define $\sigma \colon S \rightarrow \mathsf{E}(S)$, the {\em support operator}, by 
$$\sigma (s) = \overline{\phi (s)}s^{-1}s.$$
The idempotent $\sigma (s)$ is called the {\em support of $s$}.

\begin{remark}{\em  We shall mainly be interested in the value $\sigma (s)$ where $s$ is a unit in which case $s^{-1}s = 1$
and so $\sigma (s) = \overline{\phi (s)}$.}
\end{remark}

We need the following notation to state our next result.
If $Y$ is a subset of a topological space then $\mathsf{cl}(Y)$ denotes the {\em closure} of that subset.
The following shows that the support operator in the algebraic sense is a reflection of the support in the usual topological sense.

\begin{proposition}\label{prop:tempest} Let $S$ be a fundamental Boolean inverse $\wedge$-monoid.
For each unit $g$, we have that
$$U_{\sigma (g)} = \mathsf{cl}(\{ F \colon F \in \mathsf{X}(S) \mbox{ and } gFg^{-1} \neq F   \}).$$
\end{proposition}
\begin{proof} Put 
$$Y = \{ F \in \mathsf{X}(S) \colon \mbox{ and }  gFg^{-1} \neq F   \}.$$ 
If $F \in Y$ then by Lemma~\ref{lem:blue}, we have that $\phi (g) \notin F$ and so, since $F$ is a prime filter, we have that $\sigma (g) \in F$.
This shows that $Y \subseteq U_{\sigma (g)}$ and so $\mathsf{cl}(Y) \subseteq U_{\sigma (g)}$.
Let $F \in U_{\sigma (g)}$.
We show that every open set containing $F$ intersects $Y$
and it is enough to restrict attention to those open sets $U_{e}$ where $e \in F$.
Suppose that $Y \cap U_{e} = \emptyset$.
Then for every $G \in U_{e}$ we have that $gGg^{-1} = G$.
By Lemma~\ref{lem:blue}, it follows that $e \leq \phi (g)$.
But then $\phi (g), \sigma (g) \in F$, which is a contradiction. 
It follows that  $Y \cap U_{e} \neq \emptyset$, as required.
\end{proof}

\begin{lemma}\label{lem:cooper} Let $S$ be a Boolean inverse $\wedge$-monoid.
Then
$$s = \phi (s) \vee s \sigma (s)$$
is an orthogonal join, and $\phi (s \sigma (s)) = 0$.
\end{lemma}
\begin{proof} Let $s \in S$.
Observe that $1 = \phi (s) \vee \overline{\phi (s)}$.
Multiplying on the right by $s^{-1}s$ and observing that $\phi (s) \leq s^{-1}s$,
we get that $s^{-1}s = \phi (s) \vee \sigma (s)$.
Multiplying on the left by $s$ and observing that $s \phi (s) = \phi (s)$,
we get that $s = \phi (s) \vee s \sigma (s)$.
The final claim is immediate by part (3) of Lemma~\ref{lem:fpo}.
\end{proof}

\begin{lemma}\label{lem:jerry} Let $S$ be a Boolean inverse $\wedge$-monoid.
Let $g$ and $h$ be units.
\begin{enumerate}

\item $\sigma (g) = 0$ if and only if $g = 1$.

\item $\sigma (g^{-1}) = \sigma (g)$.

\item $\sigma (gh) \leq \sigma (g) \vee \sigma (h)$.

\item $\sigma (ghg^{-1}) = g \sigma (h) g^{-1}$. 

\item $\sigma (g^{2}) \leq \sigma (g)$.

\item If $\sigma (g)\sigma (h) = 0$ then $[g,h] = 1$.

\end{enumerate}
\end{lemma}
\begin{proof}  (1) One direction is immediate since $\phi (1) = 1$.
To prove the converse, suppose that $\phi (g) = 1$.
Then $1 \leq g$ and so $g = 1$.

(2) For any idempotent $e$ we have that $e \leq g$ if and only if $e \leq g^{-1}$.
It follows that $\phi (g) = \phi (g^{-1})$.

(3) From $\phi (g) \leq g$ and $\phi (h) \leq h$ we get that $\phi (g)\phi (h) \leq gh$.
Thus $\phi (g)\phi (h) \leq \phi (gh)$.
The result now follows by taking complements.

(4) From $\phi (h) \leq h$ we get that $g \phi (h) g^{-1} \leq ghg^{-1}$.
From $\phi (ghg^{-1}) \leq ghg^{-1}$ we get that $g^{-1} \phi (ghg^{-1})g \leq h$ and so $g^{-1} \phi (ghg^{-1})g \leq \phi (h)$.
Thus $\phi (ghg^{-1}) \leq g \phi (h) g^{-1}$.
It follows that $g \phi (ghg^{-1})g^{-1} = \phi (ghg^{-1})$.
Take complements to get the desired result.

(5) This is immediate by (3).

(6) From $1 =  \phi (g) \vee \overline{\phi (g)}$ and $g = 1g1$,
it quickly follows that
$$g = \overline{\phi (g)} g \overline{\phi (g)} \vee \phi (g).$$
In addition, standard Boolean algebra shows that
$\overline{\phi (g)}  \leq \phi (h)$ and $\overline{\phi (h)} \leq \phi (g)$.
We calculate
$$gh = \phi (g) \phi (h) \vee \overline{\phi (h)} h \overline{\phi (h)} \vee \overline{\phi (g)} g \overline{\phi (g)}.$$
By symmetry, this is equal to $hg$.
\end{proof}

Part (6) of Lemma~\ref{lem:jerry} is of fundamental importance to our calculations later.

The following lemma provides an avenue for deciding whether a unit is trivial or not.

\begin{lemma}\label{lem:fudge} Let $S$ be a fundamental Boolean inverse $\wedge$-monoid.
\begin{enumerate}
\item The unit $g$ is the identity if and only if $g$ fixes the set $\sigma(g)^{\downarrow}$ pointwise under the natural action.
\item Let $g$ be a non-trivial unit. Then for each non-zero idempotent $e \leq \sigma (g)$ there exists an idempotent $f \leq e$ such that $gfg^{-1} \neq f$.
\end{enumerate}
\end{lemma}
\begin{proof} (1) Only one direction needs proving.
So suppose that  $g$ fixes the set $\sigma(g)^{\downarrow}$ pointwise under the natural action.
Let $e$ be an arbitrary idempotent.
Then $e = \phi (g)e \vee \sigma (g)e$ an orthogonal join.
Put $e' = \phi (g)e$ and $f = \sigma (g)e$.
Then $geg^{-1} = ge'g^{-1} \vee gfg^{-1}$.
By Lemma~\ref{lem:blue}, we have that $ge'g^{-1} = e'$ and by assumption $gfg^{-1} = f$.
Thus $geg^{-1} = e$.
By Proposition~\ref{prop:fc}, it follows that $g = 1$.

(2)  Let $e \leq \sigma (g)$.
Suppose that $g$ fixes $e^{\downarrow}$ pointwise under conjugation.
Then, in particular, $ge = eg$ so that $ge \in \mathsf{U}(eSe)$.
But $ge$ fixes all the idempotents in $e^{\downarrow}$ pointwise under conjugation
and $eSe$ is fundamental by Lemma~\ref{lem:pros}.
Thus by Proposition~\ref{prop:fc} it follows that $ge = e$.
Since $e \leq g$ we have that $e \leq \phi (g)$.
But then $e \leq \sigma (g)\phi (g) = 0$, which is a contradiction.
\end{proof}

\subsection{The axioms}

We now state three axioms  that will play a crucial r\^ole in proving the spatial realization theorem.
They are translations (and modifications) into our language of those given in \cite[Definition 3.1]{Matui13}.
See also \cite{Fremlin}.

\begin{itemize}

\item[{\rm (F1)}] {\em Enough special involutions. }For each non-zero idempotent $e$ there are a finite number of special involutions $t_{1}, \ldots, t_{m}$ such that
$e = \bigvee_{i=1}^{m} \sigma (t_{i})$.

\item[{\rm (F2)}] {\em Shrinking.} For each involution $t$ and non-zero idempotent $e \leq \sigma (t)$ there exists a special involution $g$
such that 
$\sigma (g) \leq \mathbf{e}(te)$
and 
$\sigma (g) \leq \phi (tg)$.

\item[{\rm (F3)}] {\em Enough special $3$-cycles.} For each non-zero idempotent $e$, there exists a special $3$-cycle $g$ such that $\sigma (g) \leq e$.

\end{itemize}

\begin{proposition}\label{prop:hope} 
In a simple Tarski monoid, the axioms (F1), (F2) and (F3) all hold.
\end{proposition}
\begin{proof} 

(F1) holds. Let $e \neq 0$ be any idempotent and let $F \subseteq \mathsf{E}(S)$ be an ultrafilter containing $e$.
By Lemma~\ref{lem:george}, there is an infinitesimal $a$ such that $a \in eSe$ and $a^{-1}a \in F$.
By Lemma~\ref{lem:spooks}, the element 
$$t = a \vee a^{-1} \vee \overline{\mathbf{e}(a)}$$
is a special involution and $\phi (t) = \overline{\mathbf{e}(a)}$
by parts (4) and (5) of Lemma~\ref{lem:fpo}.
Thus $\sigma (t) = \mathbf{e}(a) \leq e$.
Since $F$ is an ultrafilter in $\mathsf{E}(S)$ either $\phi (t) \in F$ or $\sigma (t) \in F$. 
But we cannot have $\phi (t) \in F$ since then $\mathbf{d}(a), \overline{\mathbf{e}(a)} \in F$ which would give $0 \in F$.
It follows that $\sigma (t) \in F$, as required.
Denote by $I$ the set of all special involutions $t$ such that $\sigma (t) \leq e$.
We have proved that $U_{e} = \bigcup_{t \in I} U_{\sigma (t)}$.
By compactness, there is a finite subset of $I$ that does the job.
Axiom (F1) now follows.

(F2) holds.
Let $t$ be an involution and let $0 \neq e \leq \sigma (t)$.
Since $t$ is not trivial, there exists $e' \leq e$ such that $e' \neq te't$ by Lemma~\ref{lem:fudge}.
By Lemma~\ref{lem:tea-time} there is an idempotent $f \leq e'$ such that $f \perp tft$.
Thus $tf$ is an infinitesimal.
The element 
$$g = tf \vee ft \vee \overline{\mathbf{e}(tf)}$$
is a special involution where $\sigma (g) = \mathbf{e}(tf) \leq \mathbf{e}(te)$.
Observe that $tg = \mathbf{e}(tf) \vee t \, \overline{\mathbf{e}(tf)}$.
Thus $\sigma (g) \leq tg$ and so $\sigma (g) \leq \phi (tg)$.

(F3) holds. Let $e$ be a non-zero idempotent.
By Lemma~\ref{lem:horsa}, we may find infinitesimals $a,b \in eSe$ such that $(b,a)$ is a $2$-infinitesimal.
Put $g = a \vee b \vee (ba)^{-1} \vee \overline{\mathbf{e}(b,a)}$, which is a special $3$-cycle by Lemma~\ref{lem:julie},
and observe that $\sigma (g) = \mathbf{e}(b,a) \leq e$.
\end{proof}

\subsection{Local subgroups}

{\em Let $S$ be a Tarski monoid.
Throughout this section we let $G$ be any subgroup of $\mathsf{U}(S)$ which contains $\mathsf{Sym}(S)$.}

For each idempotent $e \in S$ define
$$U(e) = \{g \in G \colon \sigma (g) \leq e \},$$
called $U(e)$ the {\em local subgroup at $e$}.\\

\begin{lemma}\label{lem:tom} Let $S$ be a simple Tarski monoid.
\begin{enumerate}

\item $U(e)$ is a  subgroup of $G$.

\item $e \leq f$ if and only if $U(e) \subseteq U(f)$.

\end{enumerate}
\end{lemma}
\begin{proof} (1) This follows by parts (1), (2) and (3) of Lemma~\ref{lem:jerry}.

(2)  Only one direction needs proving.
Let $U(e) \subseteq U(f)$.
We use Lemma~\ref{lem:basic} and prove $e \leq f$ by showing that $e\bar{f} = 0$.
Suppose that $e \bar{f} \neq 0$.
By (F3), there exists a special $3$-cycle $t$ such that $\sigma (t) \leq e \bar{f}$.
Clearly, $t \in U(e)$.
If $t \in U(f)$ then $\sigma (t) \leq f$.
Thus $\sigma (t) = 0$ which implies that $t = 1$, a contradiction since $t$ is non-trivial.
It follows that $t \notin U(f)$, which is a contradiction.
\end{proof}

The following definitions are from \cite{Matui13}.
Let $t$ be a fixed involution in $G$.
Denote by $C_{t}$ the centralizer of $t$ in $G$.
Define\footnote{We have changed the notation from that in \cite{Matui13} to avoid a clash with that used for the group of units.} 
$$Z_{t} = \{ s \in C_{t} \colon s^{2} = 1, \, (\forall a \in C_{t}) [s,asa^{-1}] = 1 \}.$$
Define
$$S_{t} = \{ a^{2} \colon a \in G, (\forall s \in Z_{t}) [a,s] = 1\}.$$
Define
$$W_{t} = \{a \in G \colon  (\forall b \in S_{t})    [a,b] = 1 \}.$$
Clearly, $t \in C_{t}$, and $t \in Z_{t}$.

The following portmanteau lemma does most of the heavy lifting needed in this paper and translates Matui's proofs into algebraic language.

\begin{lemma}\label{lem:Brexit} Let $S$ be a simple Tarski monoid and let $t$ be an involution in $G$.
\begin{enumerate}

\item If $s$ is an involution in $G$ and $f$ an idempotent such that $f \leq \phi (t)$, where $sfs \neq f$, then there exists a special $3$-cycle $a \in C_{t}$ such that $[s,asa^{-1}] \neq 1$.

\item  If $s \in Z_{t}$ then $\sigma (s) \leq \sigma (t)$.

\item Let $0 \neq e \leq \sigma (t)$.
Then there exists $s \in Z_{t}$ such that 
$\sigma (s) \leq \mathbf{e}(te)$ and $\sigma (s) \leq \phi (ts)$.

\item Let $e$ be any non-zero idempotent such that $e \sigma (t) = 0$.
Then there exists a special $3$-cycle $a$ such that $\sigma (a) \leq e$ and $a^{2} \in S_{t}$.

\item Let $a \in G$ commute with every element in $Z_{t}$.
If $\sigma (t) \in F$ is an ultrafilter such that $tFt \neq F$ then $a^{2}Fa^{-2} = F$.

\item If $b \in S_{t}$ then $\sigma (t) \leq \phi (b)$.

\end{enumerate}
\end{lemma}
\begin{proof}  (1) Let $f \leq \phi (t)$ be such that $sfs \neq f$.
Then by Lemma~\ref{lem:tea-time},  there exists $e \leq f$ such that $e \perp ses$.
But $f \perp \sigma (t)$ by Lemma~\ref{lem:basic} and so $e \perp \sigma (t)$. 
By (F3), there exists a special $3$-cycle $a$ such that $\sigma (a) \leq e$.
Thus $\sigma (a) \sigma (t) = 0$.
But then $at = ta$ by Lemma~\ref{lem:jerry} and so $a \in C_{t}$.
Observe that from $e(ses) = 0$ and $\sigma (a) \leq e$, we have that $\sigma (a)s \sigma (a)s = 0$. 
Thus  $\sigma (a)\sigma (sas) = 0$ by Lemma~\ref{lem:jerry}.
It follows that $\sigma (sas) \leq \phi (a)$ by lemma~\ref{lem:basic}
We now prove that $s$ and $asa^{-1}$ do not commute.
Thus we prove that $sasa^{-1} \neq asa^{-1}s$.
By Lemma~\ref{lem:jerry} $\sigma (a^{2}) \leq \sigma (a)$.
Since $a^{2}$ is not the identity, by Lemma~\ref{lem:fudge} there is a non-zero idempotent $k \leq \sigma (a^{2})$ such that  $a^{2}ka^{-2} \neq k$.
Clearly, $aka^{-1} \neq k$.
The idempotents $k$ and $aka^{-1}$ and $a^{-1}ka$ are distinct because
if $aka^{-1} = a^{-1}ka$ then we would have $a^{2}ka^{-2} = k$.
We calculate first  $sasa^{-1}kasa^{-1}s$.
From $k \leq \sigma (a)$ we get that $sa^{-1}kas \leq \sigma (sas)$.
But then $sa^{-1}kas \leq \phi (a)$.
It follows that $asa^{-1}kasa^{-1} = sa^{-1}kas$.
Thus  
$$s(asa^{-1}kasa^{-1})s = s( sa^{-1}kas)s = a^{-1}ka.$$
We now calculate $asa^{-1}s k sasa^{-1}$.
From $\sigma (a^{-1}) = \sigma (a)$
we get that $k \leq \sigma (a^{-1})$ and so $sks \leq \sigma (sa^{-1}s)$.
It follows that $sks \leq \phi (a^{-1})$.
Thus $a^{-1}(sks)a = sks$.
It now readily follows that 
$$asa^{-1}s k sasa^{-1} = aka^{-1}.$$

(2) By part (1) above, if $s \in Z_{t}$ then $sfs = f$ for all $f \leq \phi (t)$.
It follows by Lemma~\ref{lem:blue}, that $\phi (t) \leq \phi (s)$ and so $\sigma (s) \leq \sigma (t)$, as required.

(3) By (F2), there exists a special involution $s$ satisfying $\sigma (s) \leq \mathbf{e}(te)$ and $\sigma (s) \leq \phi (ts)$.
It remains to prove that  $s \in Z_{t}$.
From $\sigma (s) \leq \phi (ts)$ we get that $\sigma (s) \sigma (ts) = 0$.
Thus $sts = tss = t$ by Lemma~\ref{lem:jerry}.
It follows that $st = ts$ and so $s \in C_{t}$.
Let  $a \in C_{t}$.
We shall prove that $s$ commutes with $b = asa^{-1}$, an involution that commutes with $t$.
Observe that
$$1 = \sigma (s)\sigma (bsb) \vee \sigma (s)\sigma (b) \vee \phi (s)\phi (bsb) \vee \phi (s) \sigma (b)$$
which follows from the fact that 
$$\sigma (bsb) \leq \sigma (b) \vee \sigma (s)
\text{ and }
\sigma (s) \leq \sigma (bsb) \vee \sigma (b)$$
by part (3) of Lemma~\ref{lem:jerry}.
Observe that $f \leq \sigma (s)$ implies $sfs = tft$ (using $\sigma (s) \leq \phi (ts)$)
and that $f \leq \phi (g)$ implies that $gfg^{-1} = f$ by part (3) of Lemma~\ref{lem:blue}.
We prove that for any idempotent $i$ we have that $(sb)i(sb) = (bs)i(bs)$
which by Proposition~\ref{prop:fc} shows that $sb = bs$, as required.
 It is enough to do this in each of the following four cases.
\begin{description}
\item[{\rm (i)}] Let $i \leq \sigma (s)\sigma (bsb)$.
Then $i, bib \leq \sigma (s)$.
Thus
$$bsisb = b(sis)b = b(tit)b = t(bib)t = s(bib)s = sb i bs.$$
\item[{\rm (ii)}] Let $i \leq \sigma (s)\sigma (b)$.
Then $i, a^{-1}ia \leq \sigma (s)$.
Thus
$$b(sis)b = b(tit)b = t(bib)t = t(asa^{-1}iasa^{-1})t = tas(a^{-1}ia)sa^{-1}t  =  tat(a^{-1}ia)ta^{-1}t = i$$
whereas
$$sbibs = sas(a^{-1}ia)sa^{-1}s =  sat(a^{-1}ia)ta^{-1}s = stits = tsist = t^{2}it^{2} = i.$$
\item[{\rm (iii)}] $i \leq \phi (s) \phi (bsb)$.
Then $i, bib \leq \phi (s)$.
Thus 
$$bsisb = bib 
\text{ whereas }
sbibs = bib.$$
\item[{\rm (iv)}] $i \leq \phi (s) \sigma (b)$.
Then $a^{-1}ia \leq \sigma (s)$.
Thus
$$bib = as(a^{-1}ia)sa^{-1} = at(a^{-1}ia)ta^{-1} = tit$$
and so
$$bsisb = bib = tit
\text{ whereas }
sbibs = stits = tsist = tit.$$
\end{description}

(4) By (F3), there exists a special $3$-cycle $a$ such that $\sigma (a) \leq e$.
From Lemma~\ref{lem:jerry} and $\sigma (a) \sigma (t) = 0$ we have that $[a,t] = 1$.
Let $s \in Z_{t}$ be arbitrary.
By part (2) above $\sigma (s) \leq \sigma (t)$ and so $\sigma (a) \sigma (s) = 0$.
It follows by Lemma~\ref{lem:jerry} that  the elements $a$ and $s$ commute.

(5) Let $a$ be a unit that commutes with every element of $Z_{t}$ and let $\sigma (t) \in F$ be an ultrafilter such that $tFt \neq F$.
We claim that $aFa^{-1} = F$ or $aFa^{-1} = tFt$.
We show first that the claim proves the result.
Observe that if $aFa^{-1} = F$ then $a^{2}Fa^{-2} = F$ and if $aFa^{-1} = tFt$ then
$$a^{2}Fa^{-2} = a (aFa^{-1})a^{-1} = a(tFt)a^{-1} = t(aFa^{-1})t = t^{2}Ft^{2} = F$$
where we use the fact that $a$ commutes with $t$.

We now prove the claim.
Suppose to the contrary that the ultrafilters 
$$F,aFa^{-1}, tFt$$ 
are distinct.
Then by Lemma~\ref{lem:soros}, there exists $e \leq \sigma (t)$, a non-zero idempotent, such that $e,tet,aea^{-1}$ are orthogonal. 
By part (3), there exists $g \in Z_{t}$ such that $\sigma (g) \leq \mathbf{e}(te) = e \vee tet$ and $\sigma (g) \leq \phi (tg)$.
Since $\sigma (g) \leq e \vee tet$ we can write $\sigma (g) = f_{1} \vee f_{2}$ where $f_{1} \perp f_{2}$ and $f_{1} \leq e$ and $f_{2} \leq tet$.
We prove that $f_{2} = 0$.
From $f_{2} \leq tet$ we get that $af_{2}a^{-1} \leq ateta^{-1} = t(aea^{-1})t$ since $a$ and $t$ commute.
From $f_{2} \leq \sigma (g)$ we get that $af_{2}a^{-1} \leq \sigma (aga^{-1}) = \sigma (g)$ since $a$ and $g$ commute.
It follows that $af_{2}a^{-1} \leq t(aea^{-1})t (e \vee tet)$.
But $taea^{-1}t \perp e \vee tet$.
Thus $af_{2}a^{-1} = 0$ and so $f_{2} = 0$.
It follows that $\sigma (g) \leq e$.
Let $0 \neq f \leq \sigma (g)$.
Then $afa^{-1} \leq aea^{-1} \perp e \vee tet$.
It follows that $afa^{-1} \leq \phi (g)$.
Thus $gafa^{-1}g = afa^{-1}$.
On the other hand $agfga^{-1} = atfta^{-1}$.
But $afa^{-1} \neq atfta^{-1}$ because if they were equal we would have that $f = tft$ which is a contradiction.
Thus $a$ and $g$ do not commute which is a contradiction.
It follows that the ultrafilters 
$$F,aFa^{-1}, tFt$$ 
are not distinct and the claim follows.

(6) Let $b \in S_{t}$ and let $\sigma (t) \in F$.
We prove that $bFb^{-1} = F$ and then apply part~(1) of Lemma~\ref{lem:blue}.
By definition $b = a^{2}$, where $a$ commutes with every element of $Z_{t}$.
If $tFt \neq F$ then $bFb^{-1} = F$ by part (5). 
Suppose, therefore, that $tFt = F$ but that $bFb^{-1} \neq F$.
By Lemma~\ref{lem:soros}, there exists $e \leq \sigma (t)$ and $e \in F$ such that $e \perp beb^{-1}$.
By Proposition~\ref{prop:tempest}, 
there is $G \in U_{e}$ such that $tGt \neq G$.
By part (5) above $bGb = G$.
But then $e$ cannot be orthogonal to $beb^{-1}$.
\end{proof}

The key result of this section is the following.
This was first proved by Matui \cite{Matui13} in the context of \'etale groupoids.

\begin{theorem}\label{prop:matui_five}
Let $S$ be a simple Tarski monoid.
If $t$ is an involution in $G$ then
$$W_{t} = U(\sigma (t)).$$
\end{theorem}
\begin{proof} We prove first that $U(\sigma (t)) \subseteq W_{t}$.
Let $g \in U(\sigma (t))$.
Then  $\sigma (g) \leq \sigma (t)$.
Let $b \in S_{t}$.
Then 
$\sigma (t) \leq \phi (b)$
by part (6) of Lemma~\ref{lem:Brexit}.
Thus $\sigma (g) \leq \phi (b)$.
It follows that $\sigma (g) \sigma (b) = 0$ by Lemma~\ref{lem:basic}.
Thus $g$ and $b$ commute by Lemma~\ref{lem:jerry}.
This implies that $g \in W_{t}$.

We now prove that $W_{t} \subseteq U(\sigma (t))$.
Let $a \in W_{t}$.
To prove that $\sigma (a) \leq \sigma (t)$ we shall prove that $a$ fixes $\phi(t)^{\downarrow}$ pointwise under conjugation and then apply Lemma~\ref{lem:blue}.
Suppose to the contrary that there is $f \leq \phi (t)$ such that $afa^{-1} \neq f$.
By Lemma~\ref{lem:tea-time}, 
there is a non-zero idempotent $e$ such that $e \leq f$ and $e(aea^{-1}) = 0$.
Observe that $e \sigma (t) = 0$ by Lemma~\ref{lem:basic}.
By part (4) of Lemma~\ref{lem:Brexit}, there exists a special $3$-cycle $b$ such that $\sigma (b) \leq e$ and $b^{2} \in S_{t}$.
Now $\sigma (b^{2}) \leq \sigma (b)$ by part (6) of Lemma~\ref{lem:jerry}.
Observe that  
$$\sigma (b^{2}) \sigma (ab^{2} a^{-1}) = \sigma (b^{2}) a \sigma (b^{2}) a^{-1} \leq \sigma (b)a\sigma (b)a^{-1}  \leq e(aea^{-1})  = 0.$$
It follows that $b^{2} \neq ab^{2}a^{-1}$ and so $ab^{2} \neq b^{2}a$ which contradicts the fact that $a \in W_{t}$ and $b^{2} \in S_{t}$.
\end{proof}

\subsection{Constructing the isomorphism: proof of Theorem~~\ref{them:matui}}

\begin{lemma}\label{lem:splott} Let $S$ and $T$ be simple Tarski monoids.
Let $G \leq \mathsf{U}(S)$ be a subgroup containing $\mathsf{Sym}(S)$ and let $H \leq \mathsf{U}(T)$  be a subgroup containing $\mathsf{Sym}(T)$
and let $\alpha \colon G \rightarrow H$ be an isomorphism. 
Let $s,t \in G$ be involutions.
Then
\begin{enumerate}

\item $\sigma (t) \leq \sigma (s)$ if and only if $\sigma (\alpha (t)) \leq \sigma (\alpha (s))$.

\item $\sigma (t) \perp \sigma (s)$ if and only if $\sigma (\alpha (t)) \perp \sigma (\alpha (s))$.

\end{enumerate}
\end{lemma}
\begin{proof} (1) Let $\sigma (t) \leq \sigma (s)$.
Then  $U(\sigma (t)) \leq U(\sigma (s))$ by Lemma~\ref{lem:tom}.
Thus $W_{t} \leq W_{s}$ by Proposition~\ref{prop:matui_five}.
By algebra,  $\alpha (W_{t}) \leq \theta (W_{s})$ and so $W_{\alpha (t)} \leq W_{\alpha (s)}$.
Thus $U(\sigma (\alpha (t))) \leq U(\sigma (\alpha (s)))$ by Proposition~\ref{prop:matui_five}.
Hence  $\sigma (\alpha (t)) \leq \sigma (\alpha (s))$ by Lemma~\ref{lem:tom}.
The reverse implication follows since $\alpha^{-1}$ is an isomorphism, and the result now follows.

(2) Suppose that $\sigma (t) \sigma (s) \neq 0$.
By (F1), there exists a special involution $r$ such that $\sigma (r) \leq \sigma (t)\sigma (s)$.
By Lemma~\ref{lem:tom}, we have that
$U(\sigma (r)) \leq U(\sigma (t)) \cap U(\sigma (s))$.
By part (1), we have that $U(\sigma (\alpha (r))) \leq U(\sigma (\alpha (t))) \cap U(\sigma (\alpha (s)))$,
and so by Lemma~\ref{lem:tom}, we have that $\sigma (\alpha (r)) \leq \sigma (\alpha (t)) \sigma (\alpha (s))$.
In particular, $\sigma (\alpha (t)) \sigma (\alpha (s)) \neq 0$.
\end{proof}

Let $S$ be a simple Tarski monoid and $G$ a subgroup of the group of units of $S$ containing $\mathsf{Sym}(S)$.
Let $F \subseteq \mathsf{E}(S)$ be a proper filter. 
Denote by $T(F)$ the set of all involutions $t \in G$ such that $\sigma (t) \in F$.
Define
$$F^{\sigma} = \{ \sigma (t) \colon t \in T(F) \},$$
the {\em support skeleton} of $F$.

\begin{lemma}\label{lem:susan} Let $S$ be a simple Tarski monoid. 
Let $F \subseteq \mathsf{E}(S)$ be a proper filter.
\begin{enumerate}
\item $F^{\sigma}$ is a filter base for $F$ and $F = (F^{\sigma})^{\uparrow}$.
\item $F$ is an ultrafilter if and only if $\sigma (t) \notin F^{\sigma}$ implies that there exists $\sigma (s) \in F^{\sigma}$
such that $\sigma (t) \perp \sigma (s)$.
\item Let $A$ be a non-empty set of idempotents satisfying the following conditions:  each element of $A$ is of the form $\sigma (t)$ for some involution $t$,
$A$ is down-directed and $\sigma (t) \notin A$ implies that there exists $\sigma (s) \in A$ such that $\sigma (t) \perp \sigma (s)$.
Then $A^{\uparrow}$ is an ultrafilter.
\end{enumerate}
\end{lemma}
\begin{proof} (1) Let $e_{1}, e_{2} \in F^{\sigma}$.
Then the product $e_{1} e_{2}$ is non-zero since the idempotents belong to the proper filter $F$.
Thus by condition (F1), there exists a special involution $t$ such that $\sigma (t)  \in F$ and $\sigma (t) \leq e_{1} e_{2}$.
It follows that $F^{\sigma}$ is a  filter base.
Clearly, $F^{\sigma} \subseteq F$.
To prove the reverse inclusion, let $e \in F$.
Then by condition (F1) there exists a special involution $t$ 
such that $\sigma (t) \in F$ and $\sigma (t) \leq e$.
By definition $\sigma (t) \in F^{\sigma}$.
Thus  $F = (F^{\sigma})^{\uparrow}$.

(2) Suppose that $F$ is an ultrafilter and  $\sigma (t) \notin F^{\sigma}$.
Then there exists $e \in F$ such that $\sigma (t) \perp e$.
By (F1), we can write $e = \bigvee_{i=1}^{m} \sigma (s_{i})$ where the $s_{i}$ are special involutions.
Since $F$ is an ultrafilter, we have that $\sigma (s_{i}) \in F$ for some $i$.
But $\sigma (t) \perp \sigma (s_{i})$ and $\sigma (s_{i}) \in F^{\sigma}$.
We now prove the converse.
We prove that $F$ is an ultrafilter.
Suppose that $e \notin F$.
By (F1), we can write $e = \bigvee_{i=1}^{m} \sigma (s_{i})$.
Clearly, $\sigma (s_{1}), \ldots, \sigma (s_{m}) \notin F$ and so none belongs to $F^{\sigma}$.
By assumption, for each $i$ we can find $\sigma (t_{i}) \in F^{\sigma}$ such that $\sigma (s_{i}) \perp \sigma (t_{i})$.
By part (1), there is $\sigma (t) \in F^{\sigma}$ such that $\sigma (t) \leq \sigma (t_{1}) \ldots \sigma (t_{m})$.
Thus $\sigma (t) \perp \sigma (s_{i})$ for all $i$ and so $\sigma (t) \perp e$.
Clearly, $\sigma (t) \in F$.

(3) Let $e, f \in A^{\uparrow}$.
Then $e' \leq e$ and $f' \leq f$ for some $e',f' \in A$.
But $A$ is down-directed and so there is $i \in A$ such that $i \leq e',f'$.
It follows that $i \leq e,f$ and so $A^{\uparrow}$ is down-directed and, of course,  closed upwards.
It follows that $A^{\uparrow}$ is a proper filter.
We prove that it is an ultrafilter.
Suppose that $e \notin A^{\uparrow}$.
By (F1), we may write $e = \bigvee_{i=1}^{m} \sigma (s_{i})$ where the $s_{i}$ are special involutions.
It follows that $\sigma (s_{i})  \notin A^{\uparrow}$ for all $i$.
By assumption, for each $i$ there exists $\sigma (t_{i}) \in A$ such that $\sigma (s_{i}) \perp \sigma (t_{i})$.
Since $A$ is down-directed, there exists $\sigma (t) \in A$ such that $\sigma (t) \leq \sigma (t_{1}) \ldots \sigma (t_{m})$.
It follows that $\sigma (t) \perp e$. 
\end{proof}

\begin{proposition}\label{them:matui_main} Let $S$ and $T$ be simple Tarski monoids. 
Let $G$ be a subgroup of the group of units of $S$ containing $\mathsf{Sym}(S)$,
let $H$ be a subgroup of the group of units of $T$ containing $\mathsf{Sym}(T)$,
and let $\alpha \colon G  \rightarrow H$ be an isomorphism of groups.
Then there exists a homeomorphism $\beta \colon \mathsf{X}(S) \rightarrow \mathsf{X}(T)$ of structure spaces such that 
 $\beta (gFg^{-1}) = \alpha (g) \beta (F) \alpha (g)^{-1}$ for every $F \in \mathsf{X}(S)$ and $g \in G$. 
\end{proposition}
\begin{proof}
We begin by constructing the function $\beta$.
Let $F \in \mathsf{X}(S)$.
Define
$$\beta (F) = \{\sigma (\alpha (t)) \colon t \in T(F) \}^{\uparrow}.$$
It follows by Lemma~\ref{lem:susan} and Lemma~\ref{lem:splott} that $\beta (F)$ is an ultrafilter.
It also follows that $\beta$ is a bijection.
By construction,  for each involution $t \in G$, we have that $\beta (U_{\sigma (t)}) = U_{\sigma (\alpha (t))}$.
By (F1), it now follows that $\beta$ is a homeomorphism.
We finish off by checking that the stated property holds.
Let $g \in G$ and $F \in \mathsf{X}(S)$.
Let $\sigma (t) \in F$ be an element of $F^{\sigma}$.
Then $\sigma (gtg^{-1}) \in gFg^{-1}$ and we observe that $gtg^{-1}$ is also an involution.
Thus $\sigma (\alpha (gtg^{-1})) \in \beta (gFg^{-1})$.
Therefore $\alpha (g) \sigma (\alpha (t)) \alpha (g)^{-1} \in \beta (gFg^{-1})$.
It is now straightforward to see that $\alpha (g) \beta (F) \alpha (g)^{-1} \subseteq \beta (gFg^{-1})$.
But equality now follows since both are ultrafilters.
\end{proof}

\begin{corollary}\label{cor:cold}  Let $S$ and $T$ be simple Tarski monoids. 
Let $G$ be a subgroup of the group of units of $S$ containing $\mathsf{Sym}(S)$,
let $H$ be a subgroup of the group of units of $T$ containing $\mathsf{Sym}(T)$,
and let $\alpha \colon G  \rightarrow H$ be an isomorphism of groups.
Then there is an isomorphism of Boolean algebras $\gamma \colon \mathsf{E}(S) \rightarrow \mathsf{E}(T)$
such that the following hold:
\begin{enumerate}

\item $\gamma (\sigma (t)) = \sigma (\alpha (t))$ for each involution $t$ in $G$ 

\item $\gamma (geg^{-1}) = \alpha (g) \gamma (e) \alpha (g)^{-1}$
for all $g \in G$ and $e \in \mathsf{E}(S)$.

\end{enumerate}
\end{corollary}
\begin{proof} By Proposition~\ref{them:matui_main}, there exists a homeomorphism 
$\beta \colon \mathsf{X}(S) \rightarrow \mathsf{X}(T)$ of structure spaces such that  $\beta (gFg^{-1}) = \alpha (g) \beta (F) \alpha (g)^{-1}$ for every $F \in \mathsf{X}(S)$ and $g \in G$. 
In particular, $\beta$ maps clopen sets of $\mathsf{X}(S)$ to clopen sets of $\mathsf{X}(T)$.
We define $\gamma$ to be this function.
This clearly defines an isomorphism of Boolean algebras.
From the fact that $\beta (U_{\sigma (t)}) = U_{\sigma (\alpha (t))}$,
we have that $\gamma (\sigma (t)) = \sigma (\alpha (t))$, which proves (1).
By (F1), it is enough to prove (2) for the case where $e = \sigma (t)$ where $t \in G$ is an involution.
But this follows from 
$$\beta (U_{\sigma (gtg^{-1})}) = U_{\gamma (\sigma (gtg^{-1}))}$$
and properties.
\end{proof}

\begin{center}
{\bf Proof of Theorem~\ref{them:matui}}
\end{center}

Let $S$ and $T$ be simple Tarski monoids.
Let $G = \mathsf{Sym}(S)$ and $H = \mathsf{Sym}(T)$ or 
let
$G = \mathsf{U}(S)$ and $H = \mathsf{U}(T)$.
Let $\alpha \colon G \rightarrow H$ be an isomorphism of groups.
By Corollary~\ref{cor:cold}, there is an isomorphism of Boolean algebras $\gamma \colon \mathsf{E}(S) \rightarrow \mathsf{E}(T)$ such that
$\gamma (geg^{-1}) = \alpha (g) \gamma (e) \gamma (g)^{-1}$ for all $g \in G$ and $e \in \mathsf{E}(S)$.
By Proposition~\ref{prop:isomorphism} and Proposition~\ref{prop:dory}, there is therefore an isomorphism $\Theta \colon S \rightarrow T$
extending both $\alpha$ and $\gamma$. \\

We conclude with an observation.
There are striking parallels between the results of this paper and those to be found in \cite{HR} on AW*-algebras.



\begin{thebibliography}{99}

\bibitem{Fremlin} D. H. Fremlin, {\em Measure algebras, Volume three, Part II}, 2nd Edition, 2012, Torres-Fremlin, Printed by lulu.com.

\bibitem{GH} S.~Givant, P.~Halmos, {\em Introduction to Boolean algebras}, Springer, 2009.

\bibitem{HR} Ch. Heunen, M. L. Reyes, Active lattices determine AW*-algebras, {\em J. Math. Anal.  Appl.} {\bf 416} (2014), 289--313.

\bibitem{PJ} P. T. Johnstone, {\em Stone spaces}, CUP, 1982.

\bibitem{KS} N. K. Kroshko, V. I. Sushchansky, Direct limits of symmetric and alternating groups with strictly diagonal embeddings, {\em Arch. Math.} {\bf 71} (1998), 173--182.

\bibitem{KLLR} G. Kudryavtseva, M. V. Lawson, D. H. Lenz, P. Resende, Invariant means on Boolean inverse monoids, {\em Semigroup Forum} {\bf 92} (2016), 77--101.

\bibitem{Kumjian} A. Kumjian, On localization and simple $C^{\ast}$-algebras, {\em Pacific J. Math.} {\bf 112} (1984), 141--192.

\bibitem{LN} Y. Lavrenyuk, V. Nekrashevych, On classification of inductive limits of direct products of alternating groups, {\em J. Lond. Math. Soc.} {\bf 75} (2007), 146--162.

\bibitem{Lawson98} M.~V.~Lawson, {\em Inverse semigroups: the theory of partial symmetries}, World Scientific, 1998.

\bibitem{Lawson07b} M.~V.~Lawson, The polycyclic monoid $P_{n}$ and the Thompson groups $V_{n,1}$, {\em Comm. Algebra} {\bf 35} (2007), 4068--4087.

\bibitem{Lawson10b} M.~V.~Lawson, A non-commutative generalization of Stone duality, {\em J. Aust. Math. Soc.} {\bf 88} (2010), 385--404.

\bibitem{Lawson12} M.~V.~Lawson, Non-commutative Stone duality: inverse semigroups, topological groupoids and $C^{\ast}$-algebras,  {\em IJAC} {\bf 22} (2012), 1250058, (47 pages).

 \bibitem{Lawson16} M.~V.~Lawson, Subgroups of the group of homeomorphisms of the Cantor space and a duality between a class of inverse monoids and a class of Hausdorff \'etale groupoids, {\em J. Algebra} {\bf 462} (2016), 77--114.

\bibitem{LL} M.~V.~Lawson,  D.~H.~Lenz, Pseudogroups and their \'etale groupoids, {\em Adv. in Math.} {\bf 244} (2013), 117--170.

\bibitem{LS} M.~V.~Lawson, P. Scott,  AF inverse monoids and the structure of countable MV-algebras, {\em J. Pure Appl. Alg.} {\bf 221} (2017), 45--74.

\bibitem{Leech} J.~Leech, Inverse monoids with a natural semilattice ordering, {\em Proc. Lond. Math. Soc.} (3) {\bf 70} (1995), 146--182.

\bibitem{Lenz} D.~H.~Lenz, On an order-based construction of a topological groupoid from an inverse semigroup, \emph{Proc. Edinb. Math. Soc.} {\bf 51} (2008), 387--406.

\bibitem{Matui12} H.~Matui, Homology and topological full groups of \'etale groupoids on totally disconnected spaces, 
{\em Proc. Lond. Math. Soc.} (3) {\bf 104} (2012), 27--56.

\bibitem{Matui13} H.~Matui, Topological full groups of one-sided shifts of finite type, {\em J. Reine Angew. Math.} {\bf 705} (2015), 35--84.

\bibitem{Nek} V.~Nekrashevych, Simple groups of dynamical origin, arXiv:1511.08241v1.

\bibitem{Renault} J.~Renault, {\it A groupoid approach to $C^*$-algebras},  Lecture Notes in Mathematics,  {\bf 793}, Springer, 1980.

\bibitem{Res1} P.~Resende, Lectures on \'etale groupoids, inverse semigroups and quantales, lecture notes for the GAMAP IP Meeting, Antwerp, 4-18 September, 2006, 115 pp.

\bibitem{Res2} P.~Resende, Etale groupoids and their quantales, {\em Adv. Math.} {\bf 208} (2007), 147--209.

\bibitem{Rubin} M. Rubin, On the reconstruction of topological spaces from their groups of homeomorphisms, {\em Trans Amer. Math. Soc.} {\bf 312} (1989), 487--538.

\bibitem{Wehrung} F. Wehrung, Refinement monoids, equidecomposability types, and Boolean inverse semigroups, 216pp, 2016, $<$hal-01197354v2$>$.

\end{thebibliography}
\end{document}